\documentclass[11pt, a4paper]{amsart}

\usepackage{amsmath}
\usepackage{amstext}
\usepackage{amssymb}
\usepackage{amsthm}
\usepackage{enumerate}
\usepackage{enumitem}
\usepackage{mathrsfs}
\usepackage[all]{xy}
\usepackage{datetime}

\usepackage[svgnames]{xcolor}
\usepackage{tikz}

\pgfdeclarelayer{edgelayer}
\pgfdeclarelayer{nodelayer}
\pgfsetlayers{edgelayer,nodelayer,main}

\tikzstyle{none}=[inner sep=0pt]

\makeatletter
\def\imod#1{\allowbreak\mkern10mu({\operator@font mod}\,\,#1)}
\makeatother

\newtheorem{theorem}{Theorem}[section]

\newtheorem{lemma}[theorem]{Lemma}

\newtheorem{corollary}[theorem]{Corollary}

\theoremstyle{definition}
\newtheorem{definition}[theorem]{Definition}
\newtheorem{problem}[theorem]{Problem}

\theoremstyle{remark}

\theoremstyle{remark}

\numberwithin{equation}{section}

    \DeclareMathOperator{\lh}{lh}

    \DeclareMathOperator{\id}{Id}

    \DeclareMathOperator{\Coll}{Coll}
    \DeclareMathOperator{\MA}{MA}

    \newcommand{\restrict}{\!\upharpoonright\!}

    \newcommand{\forces}{\Vdash}

\def\P{{\mathbb P}}

\def\Q{{\mathbb Q}}

\title[Definability and almost disjoint families]{Definability and almost disjoint families}

\author{Asger T\"ornquist}

\address{Department of Mathematical Sciences, University of Copenhagen, Universitetsparken 5, 2100 Copenhagen, Denmark}

\email{asgert@math.ku.dk}


\date{\today}


\begin{document}

\maketitle

\begin{abstract}
We show that there are no infinite maximal almost disjoint (``mad'') families in Solovay's model, thus solving a long-standing problem posed by A.D.R. Mathias in 1967. We also give a new proof of Mathias' theorem that no analytic infinite almost disjoint family can be maximal, and show more generally that if Martin's Axiom holds at $\kappa<2^{\aleph_0}$, then no $\kappa$-Souslin infinite almost disjoint family can be maximal. Finally we show that if $\aleph_1^{L[a]}<\aleph_1$, then there are no $\Sigma^1_2[a]$ infinite mad families.
\end{abstract}

\section{Introduction}

{\bf (A)} In this paper, we will denote by $\omega$ the set $\{0,1,2,\ldots\}$ of non-negative integers. Recall that a family $\mathcal A\subset \mathcal P(\omega)$ of subsets of $\omega$ is called \emph{almost disjoint} if for all $x,y\in\mathcal A$, either $x=y$ or $x\cap y$ is finite. An almost disjoint family is \emph{maximal} (``mad'') if it no strictly larger almost disjoint family contains it. It is not hard to see that a countably infinite almost disjoint family \emph{never} is maximal. On the other hand, the existence of infinite mad families follows from Zorn's lemma. 

Since we may identify $\mathcal P(\omega)$ with $2^\omega=\{0,1\}^\omega$, and since $\{0,1\}^\omega$ is naturally homeomorphic to the Cantor set when given the product topology (taking $\{0,1\}$ discrete), it makes sense to ask how \emph{definable}, in terms of the topology, an infinite mad family can be. In particular, we may ask if an infinite mad family can be Borel, analytic, or projective in the sense of \cite{kechris95}.

Almost disjoint families in general, and mad families in particular, are combinatorial objects of fundamental interest in set theory. A vast body of literature exists on the subject of the \emph{possible cardinalities} of mad families, and on the \emph{definability} of mad families, see e.g. \cite{BK13, FFZ11, kunen80, mathias77, AT13}. The most fundamental result is due to Mathias, who proved in his seminal paper \emph{Happy Families} \cite{mathias77} that there are no analytic infinite mad families. Mathias also showed, assuming the existence of a Mahlo cardinal, that there is a model of ZF (Zermelo-Fraenkel set theory without Choice) in which there are no infinite mad families. However, Mathias' techniques fell just short of allowing him to answer the following fundamental question: Is there an infinite mad family in Solovay's model? Here, Solovay's model refers to the model of ZF famously constructed by Solovay in \cite{solovay70}, in which every set of reals is Lebesgue measurable, Baire measurable, and is either countable or contains a perfect set.

\medskip

The main result of this paper is an answer to Mathias' question:

\begin{theorem}\label{t.intromainthm}
There are no infinite mad families in Solovay's model.
\end{theorem}

The proof of Theorem \ref{t.intromainthm}, which can be found in \S 3, comes out of a new proof of Mathias' result that there are no analytic mad families. That proof is presented in \S \ref{s.analytic} of the paper. In fact, in \S \ref{s.analytic} we show the following more general result:

\begin{theorem}\label{t.introkappasouslin}
Assume Martin's axiom holds at some cardinal $\kappa<2^{\aleph_0}$. Then there are no infinite $\kappa$-Souslin mad families. 
\end{theorem}

In particular, if $\MA(\aleph_1)$ holds then there are no $\mathbf{\Sigma}^1_2$ mad families, a result previously obtained by very different means by Brendle and Khomskii in \cite{BK13}. More information about $\Sigma^1_2$ mad families is obtained in \S 3 where we also prove:

\begin{theorem}\label{t.intro12}\ 

{\rm (1)} If $\aleph_1^{L[a]}<\aleph_1$, where $a\in\omega^\omega$, then there are no ${\Sigma}^1_2[a]$ infinite mad families.

{\rm (2)} If $\aleph_1^{L[a]}<\aleph_1$ for all $a\in\omega^\omega$ then there are no infinite $\mathbf{\Sigma}^1_2$ mad families.
\end{theorem}

\medskip

{\bf (B)} Our notation adheres with great fidelity to the notation established in the references \cite{kechris95,jech03, kunen80}, and they form the proper background for the paper. In particular, the reader should consult \cite[section 2]{kechris95} regarding the notation we use when dealing with trees. It bears repeating that if $T$ is a tree then $[T]$ denotes the set of infinite branches through $T$, and if $t\in T$, then $T_{[t]}$ is the set of nodes in $T$ that are compatible with $t$. If $s,t\in T$, we write $s\perp t$ if $s$ and $t$ are incompatible, i.e., have no common extension in $T$. 

Throughout the paper, functions into $2=\{0,1\}$ are frequently identified with the sets that they are the characteristic functions of.

One notational peculiarity is the following: If $x$ is an ordered pair, $x=(y,z)$, we define $x_*=y$ and $x^*=z$, so that $x=(x_*,x^*)$. This notation is particularly useful when discussing trees on products of two sets.

\medskip

{\bf (C)} It was previously announced, along with the results in this paper, that the same techniques also could prove a number of parallel results about eventually different families of functions from $\omega$ to $\omega$, in particular that there are no analytic maximal eventually different families. Alas, that proof was irreparably flawed, and these questions remain open. \S 4 below contains a discussion of open problems.

\medskip

{\it Acknowledgement}: I wish to thank, above all, Katherine Thompson for many invaluable conversations about the ideas for this paper during her visits to Copenhagen and my visits to Bonn.

The research for this paper was generously supported by a Sapere Aude starting grant (level 2) from Denmark's Natural Sciences Research Council.

\section{Analytic and $\kappa$-Souslin almost disjoint families}\label{s.analytic}

In this section, we will first give a new proof of Mathias' result that there are no analytic infinite mad families, and subsequently generalize this proof to obtain a similar result about $\kappa$-Souslin mad families in the presence of Martin's axiom.

\subsection{Analytic almost disjoint families revisited}

\begin{theorem}[Mathias \cite{mathias77}]\label{t.mathias}
There are no infinite analytic maximal almost disjoint families in $\mathcal P(\omega)$.
\end{theorem}

Our new proof of Theorem \ref{t.mathias} starts with a rather innocuous lemma. We note for future use that the proof only uses ZF+DC (and even then only a small part of this).

\begin{lemma}\label{l.easy}
Let $\mathcal A$ be a family of subsets of $\omega$, and suppose there is a sequence $(A_n)_{n\in\omega}$ of subsets of $\omega$ such that
\begin{enumerate}
\item every $x\in\mathcal A$ is almost contained in the union of finitely many $A_n$;
\item for all $n$, $\omega\setminus\bigcup_{i\leq n} A_i$ is infinite.
\end{enumerate}
Then there is an infinite $z\subseteq\omega$ such that $x\cap z$ is finite for all $x\in\mathcal A$.
\end{lemma}
\begin{proof}
Let $k_0=0$, and for $n>0$ let $k_{n}$ be the least element in $\omega\setminus\bigcup_{i<n} A_i$ greater than $k_{n-1}$, and let $z=\{k_i:i\in\omega\}$. If $x\in\mathcal A$, find $n$ such that $x\subseteq^*\bigcup_{i<n} A_i$. Since $\{k_i:i\geq n\}\cap \bigcup_{i<n} A_i=\emptyset$ we have $|z\cap x|<\infty$.
\end{proof}

The strategy for proving Theorem \ref{t.mathias} is to produce a family $(A_n)$ as in Lemma \ref{l.easy} through an ordinal analysis of tree representations. The central definition is that of a diagonal sequence (Figure 1):

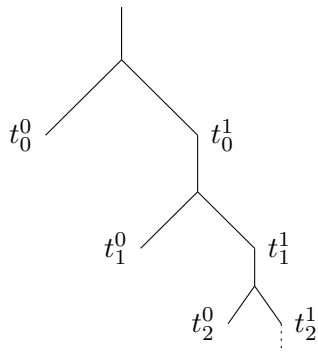
\begin{figure}\label{f.diagseq}
\centering
\begin{tikzpicture}[thick, scale=1]
	\begin{pgfonlayer}{nodelayer}
		\node [style=none] (9) at (-2, 4.7) {};
		\node [style=none] (0) at (-2, 4) {};
		\node [style=none] (1) at (-3, 3) [label=left:$t_0^0$] {};
		\node [style=none] (2) at (-1, 3) [label=right:$t_0^1$]{};
		\node [style=none] (10) at (-1,2.25) {};
		\node [style=none] (3) at (-1.75, 1.5) [label=left:$t_1^0$] {};
		\node [style=none] (4) at (-.25, 1.5) [label=right:$t_1^1$]{};
		\node [style=none] (5) at (-.25, 1) {};
		\node [style=none] (6) at (0.1, 0.5) [label=right:$t_2^1$]{};
		\node [style=none] (7) at (-0.6, 0.5) [label=left:$t_2^0$] {};
		\node [style=none] (8) at (0.1, 0.1) {};

	\end{pgfonlayer}
	\begin{pgfonlayer}{edgelayer}
		\draw (9.center) to (0.center);
		\draw (0.center) to (1.center);
		\draw (0.center) to (2.center);
		\draw (2.center) to (10.center);
		\draw (10.center) to (3.center);
		\draw (10.center) to (4.center);
		\draw (4.center) to (5.center);
		\draw (5.center) to (6.center);
		\draw (5.center) to (7.center);
		\draw[dotted] (6.center) to (8.center);
	\end{pgfonlayer}
\end{tikzpicture}
\caption{A diagonal sequence branching.} \label{fig.diagseq}
\end{figure}

\begin{definition}
Let $T$ be a tree on $2\times\omega$, $\alpha\leq\omega$. A sequence $t_i^{j}\in T$, $i<\alpha$, $j\in\{0,1\}$ is a \emph{diagonal sequence} (of length $\alpha$) in $T$ if
\begin{enumerate}
\item $\lh(t_i^0)=\lh(t_i^1)$;
\item $t_i^1\subseteq t_{i+1}^{j}$;
\item $t_i^0$ and $t_i^1$ are incompatible in the first coordinate;
\item For all $s, t\in T$ with $s\supseteq t_i^0$ and $t\supseteq t_i^1$ we have\footnote{When we write $s\cap t$ for some $s,t\in 2^{<\omega}$, we mean the intersection of the sets that $s$ and $t$ are the characteristic functions of.} $s_*\cap t_*=(t_i^0)_*\cap(t_i^1)_*$.
\item $p[T_{[t_i^0]}]\neq\emptyset$ for all $i\in\omega$.
\end{enumerate}

The diagonal sequence $(t^j_i)$ is called infinite or finite according to if $\alpha$ is infinite or finite. If $\alpha<\omega$, then we call the node $t^1_{\alpha-1}$ the \emph{top leaf} of the diagonal sequence. If a finite diagonal sequence in $T$ is extended by no strictly longer diagonal sequence then we say it is $T$-terminal. We let 
$$
\tau(T)=\{t\in T: t\text{ is top leaf of a } T\text{-terminal diagonal sequence}\}.
$$ 
\end{definition}

\begin{lemma}\label{l.diagseq}
If $T$ is a tree on $2\times\omega$ such that $p[T]$ is an uncountable almost disjoint family then $T$ admits an infinite diagonal sequence.
\end{lemma}

\begin{proof} Fix $T$. Recall that $s,r\in T$ are said to be \emph{incompatible in the first coordinate} if $s_*\perp r_*$.

\medskip

{\bf Claim.} Suppose $t,s,r\in T$ are such that $s,r\supseteq t$ and $s,r$ are incompatbile in the first coordinate. Then there are $s',t'\in T$ with $s'\supseteq s$ and $t'\supseteq t$ such that for all $s'',t''\in T$ with $s''\supseteq s'$ and $t''\supseteq t'$ we have $s''_*\cap t''_*=s'_*\cap t'_*$.

\medskip

{\it Proof of the claim}: Otherwise we can find sequences
$$
s=s_0\subseteq s_1\subseteq\cdots\subseteq s_n\subseteq\cdots
$$
and
$$
t=t_0\subseteq t_1\subseteq\cdots\subseteq t_n\subseteq\cdots
$$
in $T$ with $|{s_i}_*\cap {t_i}_*|=|{s_0}_*\cap {t_0}_*|+i$ for all $i$. Let $x=\bigcup_{i} (s_i)_*$ and $y=\bigcup_{i} (t_i)_*$. Then $x,y\in p[T]$ and $|x\cap y|=\infty$, contradicting that $p[T]$ is an almost disjoint family.\hfill {\tiny Claim.}$\dashv$

\medskip

Notice that the claim implies that if $(t_i^j)_{i\leq n}$ is a $T$-terminal finite diagonal sequence then $|p[T_{t^1_n}]|\leq 1$. (To see this, prune $T_{[t^1_n]}$ and apply the claim.)

Suppose then that there are no infinite diagonal sequences in $T$. Let $T^0=T$. If $T^\alpha$ has been defined for some ordinal $\alpha$, let
$$
T^{\alpha+1}=T^\alpha\setminus \{t\in T^{\alpha}: (\exists s\in \tau(T^\alpha))\ t\supseteq s\}.
$$
At limit stages, let $T^\lambda=\bigcap_{\alpha<\lambda} T^\alpha$. Since $T$ is countable and has no infinite diagonal sequences, we must have $T^\beta=\emptyset$ for some $\beta<\omega_1$. The claim above implies that $|p[T^\alpha_t]|\leq 1$ when $t\in\tau(T^\alpha)$, and since
$$
p[T]=\bigcup_{\alpha<\beta}\bigcup_{t\in \tau(T^\alpha)} p[T^\alpha_{[t]}],
$$
it follows that $p[T]$ is countable.
\end{proof}

\begin{proof}[Proof of Theorem \ref{t.mathias}]

Fix a tree $T$ on $2\times\omega$ such that $p[T]$ is an uncountable almost disjoint family. We will find an ordinal $\alpha^*<\omega_1$, and

\begin{enumerate}[label={(\Alph*)}]
\item for each $\alpha\leq\alpha^*$, a tree $T^\alpha\subseteq T$ (unrelated to the definition in the previous lemma!);

\item for each $\alpha<\alpha^*$, an infinite diagonal sequence $(t_i^j(\alpha))$ in $T^\alpha$;

\item sets $A^{\alpha}_i\subseteq\omega$, for $i\in\omega, \alpha\leq\alpha^*$.
\end{enumerate}
such that the following are satisfied:
\begin{enumerate}
\item $T^0=T$;

\item $A_i^\alpha=\bigcup p[T^\alpha_{[t_i^0(\alpha)]}]$;

\item for $\alpha<\alpha^*$,
\begin{align*}
T^{\alpha+1}=\{t\in T^{\alpha}:&(\exists x\in p[T^{\alpha}_{[t]}])(\forall k)(\forall \alpha_0,\ldots,\alpha_{k-1}<\alpha)\\
&(\forall i_0,\ldots, i_{k-1}\in\omega)\ x\not\subseteq^*\bigcup_{j<k} A^{\alpha_j}_{i_j}\};
\end{align*}

\item for $\lambda\leq\alpha^*$ a limit ordinal, $T^\lambda=\bigcap_{\alpha<\lambda} T^\alpha$;

\item $T^{\alpha^*}$ admits no infinite diagonal sequences.
\end{enumerate}
(In (2), the set $p[T^\alpha_{[t_i^0(\alpha)]}]$ should be thought of as a family of subsets of $\omega$, so that $A^{\alpha}_i\subseteq\omega$.)

Notice that $t^j_i(\alpha)\notin T^{\alpha+1}$. Lemma \ref{l.diagseq} and the fact that $T$ is countable implies that we can find $\alpha^*<\omega_1$, $T^\alpha$, $t^j_i(\alpha)$ and $A^{\alpha}_i$ as required by (A)--(C) and (1)--(5). Let $A^{\alpha^*}_i$ enumerate the countable set $p[T^{\alpha^*}]$, if need be letting $A^{\alpha^*}_i=\emptyset$ when $i\geq|p[T^{\alpha^*}]|$. Let $(A_i)_{i\in\omega}$ enumerate the countably infinite family $\{A^\alpha_i:\alpha\leq\alpha^*, i<\omega\}$. We claim that conditions (1) and (2) of Lemma \ref{l.easy} are satisfied.

For (1), note that if $x\in p[T]$ then either $x\in p[T^{\alpha^*}]$ (and we are done), or else there is some least $\alpha<\alpha^*$ such that $x\notin p[T^{\alpha}]$. Assume the latter is the case. Let $y\in\omega^\omega$ be such that $(x,y)\in [T]$. Since $(x,y)\notin [T^\alpha]$, there must be some $\beta<\alpha$ and $n\in\omega$ such that $(x\restrict n,y\restrict n)\in T^\beta$ but $(x\restrict n,y\restrict n)\notin T^{\beta+1}$. Then (3) gives that $x$ is covered by finitely many $A_i^{\delta}$, $\delta\leq\beta$, $i\in\omega$.

For (2), consider finitely many sets $A^{\alpha_0}_{i_0},\ldots, A^{\alpha_k}_{i_k}$. Let 
$$
\beta=\max\{\alpha_j:j\leq k\wedge\alpha_{i_j}<\alpha^*\},
$$
and let $l=\max\{i_j:j\leq k\}$. There is $x\in T^\beta_{[t^0_{l+1}]}$ which is not almost contained in $\bigcup\{A_{i_j}^{\alpha_j}:\alpha_j<\beta\}$. By definition, $x$ is almost disjoint from the sets $A^{\beta}_i$ when $i\leq l$, and $x$ is almost disjoint from all $A^{\alpha^*}_i$, $i\in\omega$. Thus
$$
x\setminus\bigcup_{j\leq k} A^{\alpha_j}_{i_j}
$$
is infinite, which shows that (2) of Lemma \ref{l.easy} holds.
\end{proof}

\subsection{$\boldsymbol{\kappa}$-Souslin almost disjoint families.} The previous proof has the following straight-forward generalization to $\kappa$-Souslin almost disjoint families.

\begin{theorem}\label{t.kappasouslin}
If $\MA(\kappa)$ holds for some $\kappa<2^{\aleph_0}$ then no $\kappa$-Souslin almost disjoint family is maximal.
\end{theorem}

We refer to \cite{kunen80} for background on MA (Martin's Axiom). The steps of the proof are almost entirely parallel to the analytic case, so we will provide only a sketch.

\begin{lemma}\label{l.souslineasy}
Assume $\MA(\kappa)$ holds for some $\kappa<2^{\aleph_0}$. Let $\mathcal A$ be a family of subsets of $\omega$, and suppose there is a sequence $(A_\vartheta)_{\vartheta\in\kappa}$ of subsets of $\omega$ such that
\begin{enumerate}
\item every $x\in\mathcal A$ is almost contained in the union of finitely many $A_\vartheta$;
\item For all $n$, if $\vartheta_0,\ldots,\vartheta_{n-1}\in\kappa$, then $\omega\setminus\bigcup_{i<n}A_\vartheta$ is infinite.
\end{enumerate}
Then there is an infinite $z\subseteq\omega$ such that $x\cap z$ is finite for all $x\in\mathcal A$.
\end{lemma}
\begin{proof}
Let $\Q$ be the forcing poset with conditions being pairs $(s,\mathcal F)$ where $s\in 2^{<\omega}$ and $\mathcal F$ a finite subsets of $\{A_\vartheta:\vartheta<\kappa\}$, and the ordering $(s,\mathcal F)\leq_{\P} (s',\mathcal F')$ just in case $s\supseteq s'$ and $\mathcal F\supseteq\mathcal F'$, and
$$
s\setminus s'\subseteq (\omega\setminus \bigcup\mathcal F').
$$
Since any two conditions which agree on the first coordinate are compatible, $\Q$ has the ccc. Notice that
$$
D_{\mathcal F'}=\{(s,\mathcal F): \mathcal F\supseteq\mathcal F'\}
$$
is dense in $\Q$ for any finite $\mathcal F'\subseteq\{A_\vartheta:\vartheta<\kappa\}$, and so is
$$
D_n=\{(s,\mathcal F): |s|\geq n\}
$$
for any $n\in\omega$. (Here $|s|$ denotes the cardinality of the set that $s$ is formally the characteristic function of.)

Since $\MA(\kappa)$ holds, there is a filter $G$ on $\Q$ meeting the $\kappa$ dense sets above. Let
$$
z=\bigcup \{s\in 2^{<\omega}: (\exists\mathcal F) (s,\mathcal F)\in G\}.
$$
This $z$ is as required.
\end{proof}

The notion of a diagonal sequence in a tree on $2\times\kappa$ is defined exactly as it was for a tree on $2\times\omega$.

\begin{lemma}\label{l.diagseqsouslin}
If $T$ is a tree on $2\times\kappa$ such that $p[T]$ is an almost disjoint family of size $>\kappa$ then $T$ admits an infinite diagonal sequence.
\end{lemma}

The proof of the this lemma is entirely parallel to that of Lemma \ref{l.diagseq}, except for the fact that in the ordinal analysis, $T^\beta=\emptyset$ will happen at some $\beta<\kappa^+$, and so the proof instead shows that if $T$ does not admit any infinite diagonal sequences then $|p[T]|\leq\kappa$.

\begin{proof}[Proof of Theorem \ref{t.kappasouslin}]
Fix a tree $T$ on $2\times\kappa$ and assume $p[T]$ is almost disjoint. It is well known that under MA$(\kappa)$ no almost disjoint family of size $\leq \kappa$ is maximal, see e.g. \cite{kunen80}. So we may assume that $|p[T]|>\kappa$.

The ordinal analysis in the proof of Theorem \ref{t.mathias} goes through unchanged using Lemma \ref{l.diagseqsouslin} instead of Lemma \ref{l.diagseq}, except that it will now end at some $\alpha^*<\kappa^+$. The rest of the proof then goes through with the only change being that Lemma \ref{l.souslineasy} is used instead of Lemma \ref{l.easy}. 
\end{proof}

\begin{corollary}
$\MA(\aleph_1)$ implies that there are no $\mathbf{\Sigma}^1_2$ mad families.
\end{corollary}
\begin{proof}
Since every $\mathbf{\Sigma}^1_2$ set is $\aleph_1$-Souslin (see \cite{moschovakis80}).
\end{proof}

More information about $\Sigma^1_2$ almost disjoint families will be obtained in the next section where we will prove Theorem \ref{t.intro12}.

\medskip

The proof of Theorem \ref{t.mathias} and \ref{t.kappasouslin} also shows the following:

\begin{corollary}
If $\mathcal A\subset\mathcal P(\omega)$ is an infinite $\kappa$-Souslin almost disjoint family, then $\mathcal A$ is contained in a proper $\kappa$-generated ideal on $\omega$. 
\end{corollary}

\section{No infinite mad families in Solovay's model}

In this section we will prove our main result:

\begin{theorem}\label{t.main1}
There are no infinite maximal almost disjoint families in Solovay's model.
\end{theorem}

Let us review the basics of Solovay's model that are needed below. (Solovay's original paper \cite{solovay70}, as well as \cite[Chapter 26]{jech03}, provide very readable accounts of the ins and outs of Solovay's model. L\'evy's collapsing poset is defined in \cite[Chapter 15]{jech03}.) Starting with the constructible universe $L$, and assuming that there is an inaccessible cardinal $\Omega\in L$, let $M_0$ be the model obtained by adding to $L$ a generic collapsing map from $\omega$ to the ordinals $<\Omega$. Solovay's model $M$ is then the class of sets in $M_0$ that are hereditarily definable from $\omega$-sequences of ordinals. Note that $\aleph_1^{M}=\aleph_1^{M_0}=\Omega$.

\medskip

{\bf Fact 1.} If $x,y\in\omega^\omega\cap M$ then there is some poset $\P$ of size $<\Omega$ and $\P$-name $\sigma$ such that for $\sigma_G=x$ for some $L[y]$-generic $G$ in $M$; in fact, $\P$ can always be chosen to be the poset $\Coll(\omega, \kappa)$ collapsing some cardinal $\kappa<\Omega$ to $\omega$.

\medskip

{\bf Fact 2.} If $A\in M$ and $A\subseteq\omega^\omega$ then there is $a\in \omega^\omega\cap M$ and a formula $\varphi$ with parameters in $L[a]$ such that
$$
x\in A\iff L[a][x]\models \varphi(x).
$$

\medskip

Fact 2 is usually referred to as the ``Solovay property''. There is, of course, nothing privileged about $\omega^\omega$: Any Polish space in $M$ will do.

\medskip

The overall strategy for the proof of Theorem \ref{t.main1} is to mimic the tree analysis proofs of \S 2, but replace the tree with a poset $\P$ and a $\P$-name for a real.

Given a poset $\P$, a $\P$-name $\sigma$, and $p\in\P$, we let
$$
x(\sigma,p)=\{n\in\omega: (\exists q\leq_\P p) q\forces \check n\in\sigma\}.
$$
The central definition is the following.

\begin{definition} 
Let $\P$ be a poset, $r\in\P$ a condition, and let $\sigma$ be a $\P$-name. By a $(\P,\sigma)$-\emph{diagonal sequence below $r$ of length $\alpha\leq\omega$}, we mean a sequence $(p_i^j, n_i)$, $i<\alpha$, $j\in\{0,1\}$, in $\P\times\omega$ such that
\begin{enumerate}
\item $r\forces_\P\sigma\subseteq\omega$;
\item $p_i^j\leq_{\P} r$ for all $i,j$;
\item $p^j_{i+1}\leq_\P p^1_i$ for all $i<\alpha-1$ and $j\in\{0,1\}$;
\item $x(\sigma,p^0_i)\cap x(\sigma, p^1_i)\subseteq n_i$.
\end{enumerate}

A $(\P,\sigma)$-diagonal sequence will called finite or infinite according to whether $\alpha$ is finite or infinite.

\end{definition}

Given a forcing notion $\P$ and a $\P$-name $\sigma$, we let $\sigma_0,\sigma_1$ be the natural $\P\times\P$-names such that if $H=H_0\times H_1$ is $\P\times\P$-generic then $(\sigma_0)_H=\sigma_{H_0}$ and $(\sigma_1)_H=\sigma_{H_1}$.

\begin{lemma}\label{l.diagseqsolovay}
Let $N$ be an inner model, $\P\in N$ a forcing notion. Let $\sigma$ be a $\P$-name in $N$ and suppose $p\in\P$ is a condition such that 
$$
p\forces \sigma\subseteq\omega\wedge \sigma\notin N
$$
and 
\begin{equation}\tag{$*$}\label{eq.diag}
(p,p)\forces_{\P\times\P} \sigma_0\neq\sigma_1\to |\sigma_0\cap\sigma_1|<\aleph_0.
\end{equation}
Then there is an infinite $(\P,\sigma)$-diagonal sequence below $p$.
\end{lemma}

\begin{proof} Fix a $\P$-generic $G$ over $N$ such that $p\in G$. For any $p'\leq p$ we must have $\sigma_G\neq x(p',\sigma)$, since otherwise $p'\forces \sigma\in N$. Now build an infinite $(\P,\sigma)$-diagonal sequence as follows: Choose $q,r\leq p$ such that for some $n\in\omega$, $q\forces \check n\in\sigma$ and $r\forces \check n\notin\sigma$. Since then $(q,r)\forces_{\P\times\P} \sigma_0\neq\sigma_1$, it follows from (\ref{eq.diag}) that we can find $n_0\in\omega$ and $p_0^0\leq q$ and $p_0^1\leq r$ such that 
$$
(p_0^0,p_0^1)\forces_{\P\times\P} \sigma_0\cap\sigma_1\subseteq \check n_0.
$$
Repeat this argument below $p_0^1$ (instead of $p$) to obtain $n_1$ and $p_1^0, p_1^1\leq p_0^1$. And so on. \end{proof}

\begin{proof}[Proof of Theorem \ref{t.main1}]

Work in Solovay's model $M$. Let $\mathcal A\subseteq \mathcal P(\omega)$ be an uncountable almost disjoint family, and let $a\in \omega^\omega$ and $\varphi$ be such that
$$
x\in \mathcal A\iff L[a][x]\models \varphi(x).
$$
Let $(\P_{\xi},r_\xi,\sigma_\xi)$ enumerate (in $L[a]$) all triples in $L[a]$ where $\P_\xi$ is the poset for collapsing some cardinal $<\Omega$ to $\omega$, $r_\xi\in \P_\xi$, and $\sigma_\xi$ is a $\P_\xi$-name such that $r_\xi\forces \varphi(\sigma_\xi)\wedge\sigma_\xi\notin L[a]$. We define, by recursion on ordinals $\alpha$ in $L[a]$, sets $A^\alpha_i\subseteq\omega$, $i\in\omega$, as follows: If $A_i^{\beta}$ have been defined for $\beta<\alpha$, let $\xi$ be least such that there is an infinite $(\P_\xi,\sigma_\xi)$-diagonal sequence $(p_i^j)$ below $r_\xi$, and
$$
r_\xi\forces_{\P_{\xi}} (\forall k<\omega)(\forall \beta_0,\ldots\beta_k<\alpha)(\forall i_0,\ldots, i_k<\omega)\ |\sigma_\xi\setminus \bigcup_{j\leq k} A_{i_j}^{\beta_j}|=\aleph_0,
$$
and let $A_i^\alpha=x(\sigma_\xi,p^0_i)$. If no such $\xi$ exists then the process stops at $\alpha$, and we let $\alpha^*=\alpha$. 

\medskip

{\bf Claim 1.} $\alpha^*>0$.

\medskip

{\it Proof of claim}: We need to show that there is at least \emph{one} triple of the form $(\P_\xi,r_\xi,\sigma_\xi)$ as described above. Since $\mathcal A$ is uncountable it is not contained in $\mathcal P(\omega)\cap L[a]$. By Fact 1, we can find $\kappa<\Omega$, a $\Coll(\omega,\kappa)$-name $\sigma$, and $G\in M$ generic over $L[a]$, such that $\sigma_G\in\mathcal A$ and $\sigma_G\notin L[a]$. Then for some $p\in \Coll(\omega,\kappa)$ we have $p\forces \varphi(\sigma)\wedge\sigma\notin L[a]$. Since clearly $(p,p)\forces \varphi(\sigma_0)\wedge\varphi(\sigma_1)$, it follows that $p$ also satisfies (\ref{eq.diag}) of Lemma \ref{l.diagseqsolovay}. Thus the triple $(\Coll(\omega,\kappa),p,\sigma)$ works.
\hfill {\tiny Claim.}$\dashv$

\medskip

{\bf Claim 2.} If $G$ is $\P_\xi$-generic for some $\xi$ and $r_\xi\in G$, then $(\sigma_\xi)_G$ is almost contained (i.e., contained modulo a finite set) in the union of finitely many $A_i^\beta$, where $\beta<\alpha^*$, $i<\omega$.

\medskip

{\it Proof of claim}: If not, then there is $q\leq_{\P_{\xi}} r_{\xi}$ such that
$$
q\forces_{\P_{\xi}}(\forall k<\omega)(\forall \beta_0,\ldots\beta_k<\alpha^*)(\forall i_0,\ldots, i_k<\omega)\ |\sigma_\xi\setminus \bigcup_{j\leq k} A_{i_j}^{\beta_j}|=\aleph_0.
$$
But this contradicts that the process stops at $\alpha^*$. \hfill {\tiny Claim.}$\dashv$

\medskip

Let now $(A_j)_{j\in\omega}$ enumerate (in $M$) the all sets $A_i^{\alpha}$, $i<\omega,\alpha<\alpha^*$ as well as all the elements of $\mathcal A\cap L[a]$. The previous claim shows that every $x\in \mathcal A$ is contained in the union of finitely many of the $A_j$, so (1) in Lemma \ref{l.easy} is satisfied. (2) is established by an argument identical to that given in the proof of Theorem \ref{t.mathias}. Whence $\mathcal A$ is not maximal.
\end{proof}

Clearly, the previous proof also gives the following two corollaries:

\begin{corollary}
Every infinite almost disjoint family in Solovay's model is contained in a countably generated proper ideal on $\omega$.
\end{corollary}

\begin{corollary}
In $M_0$, the model obtained from $L$ by collapsing all cardinals below the inaccessible $\Omega$ to $\omega$, there are no projective mad families.
\end{corollary}

The idea behind the proof of Theorem \ref{t.main1} can also be used to prove the following (a.k.a. Theorem \ref{t.intro12}), which seems to not have been known previously.

\begin{theorem}
If $\aleph_1^{L[a]}<\aleph_1$ then there are no infinite $\Sigma^1_2[a]$ mad families. 
\end{theorem}

\begin{proof}
Let $\mathcal A$ be an infinite almost disjoint family which is $\Sigma^1_2$. (The argument that follows routinely relativizes to the parameter $a$, so we omit it.) The fact that $\mathcal A$ is infinite and almost disjoint is absolute by Shoenfield's theorem; see \cite[Theorem 25.20]{jech03}.

\medskip

Let $\P=\Coll(\omega,\aleph^1_L)$, and let $\varphi(x)$ be a $\Sigma^1_2$ formula defining $\mathcal A$.

\medskip

{\bf Claim 1.} There is a sequence $(A^\beta_i)_{i\in\omega,\beta<\aleph_1^L}$ in $L$ such that whenever $G$ is $\P$-generic over $L$ then 
$$
L[G]\models (\forall x)(\varphi(x)\to (\exists\beta_0,\ldots,\beta_k<\aleph_1^L)(\exists i_0,\ldots,i_k\in\omega)\  x\subseteq^* \bigcup_{j\leq k} A_{i_j}^{\beta_j}\ )
$$
and for any finitely many $\beta_0,\ldots,\beta_k<\aleph_1^L$ and $i_0,\ldots,i_k$ we have that
$$
|\omega\setminus \bigcup_{j\leq k} A_{i_j}^{\beta_j}|=\aleph_0.
$$

\medskip

{\it Proof of claim 1}:
We may assume that $\mathcal A$ is uncountable. Since every subset of $\omega$ in $L[G]$ has the form $\sigma_G$ for some $\P$-name $\sigma$, we can proceed exactly as in the proof of Theorem \ref{t.main1} to produce the desired sequence $A_i^{\alpha}\subseteq\omega$, $i\in\omega$, $\alpha<\alpha^*<\aleph_2^L$.
\hfill {\tiny Claim.}$\dashv$

\medskip

Fix a sequence $(A^\beta_i)_{i\in\omega,\beta<\aleph_1^L}$ as in Claim 1 above. Let $\Q$ be the ccc forcing notion (in $L$) defined as in Lemma \ref{l.souslineasy}, with conditions $(s,\mathcal F)$ where $s\in 2^{<\omega}$ and $\mathcal F\subseteq\{A^\beta_i:i\in\omega,\beta<\aleph_1^L$\} is finite. If $H\subseteq\Q$ is a filter then we let
$$
x_H=\bigcup\{s:(\exists \mathcal F) (s,\mathcal F)\in H\}.
$$
(Below we will, as usual, identify $x_H$ with the subset of $\omega$ it is the characteristic function of.) Recall from Lemma \ref{l.souslineasy} that if $H$ is sufficiently generic then $x_H\in [\omega]^\omega$.

\medskip

{\bf Claim 2.} $1_\Q\forces_{\Q,L}(\forall y\subseteq\omega) \varphi(y)\to |x_H\cap y|<\infty.$

\medskip

{\it Proof of claim 2}:
Suppose, aiming for a contradiction, that there is $q\in \Q$ such that
\begin{equation}\tag{$**$}\label{eq.force}
q\forces_{\Q,L}(\exists y\subseteq\omega) \varphi(y)\wedge |x_H\cap y|=\infty.
\end{equation}
Let $G$ be $\Coll(\omega,\aleph_1^L)$-generic over $L$. Since $\Q$ has the ccc and $(|\Q|=\aleph_1)^L$, there is $H\in L[G]$ which is $\Q$-generic over $L$ and $q\in H$.  By Claim 1, in $L[G]$ we have $|x_H\cap y|<\infty$ for all $y\in\mathcal A\cap L[G]$,  contradicting (\ref{eq.force}).
\hfill {\tiny Claim.}$\dashv$

\medskip

Finally, to finish the proof, assume $\aleph_1^L<\aleph_1$. Then, again using that $\Q$ has the ccc and $(|\Q|=\aleph_1)^L$, there is $H$ (in $V$) which is $\Q$-generic over $L$. By Claim 2,
$$
L[H]\models (\forall y\subseteq\omega) \varphi(y)\to |x_H\cap y|<\infty.
$$
Since $\varphi(x)$ is a $\Sigma^1_2$ predicate,
\begin{equation}\tag{$\dag$}\label{eq.LH}
(\forall y\subseteq\omega) \varphi(y)\to |x_H\cap y|<\infty
\end{equation}
is $\Pi^1_2[x_H]$, so is absolute for models containing $x_H$. In particular, (\ref{eq.LH}) holds in $V$, whence $x_H\cap y$ is finite for all $y\in\mathcal A$.
\end{proof}

\begin{corollary}
If $\aleph_1^{L[a]}<\aleph_1$ then every infinite $\Sigma^1_2[a]$ almost disjoint family  is contained in a countably generated proper ideal on $\omega$.
\end{corollary}

\section{Some open problems}

We close the paper by discussing a number of related problems that are, to the best of the author's knowledge, still wide open. 

\medskip

In light of the results of the present paper, the most obvious problem left open is the following:

\begin{problem}
Does AD, the Axiom of Determinacy\footnote{We refer to \cite{kechris95} for an introduction to AD.}, imply that there are no infinite mad families?
\end{problem}

The author expects the answer to this to be ``yes''.

\medskip

Another kind of almost disjoint family is that of an \emph{eventually different family} in $\omega^\omega$ (where $\omega^\omega$ is given the product topology, taking $\omega$ discrete). Here, two functions $f,g\in\omega^\omega$ are said to be eventually different if there is some $N$ such that for all $n\geq N$ we have $f(n)\neq g(n)$. The notion of a maximal eventually different family in $\omega^\omega$ is the obvious one. The following problem seems to go back to Mathias and Velickovic:

\begin{problem}
Can a maximal eventually different family be analytic?
\end{problem}

In the same vein:

\begin{problem}
Are there any maximal eventually different families in Solovay's model?
\end{problem}

Embarrassingly, we don't even know if such a family can be \emph{closed}. (It follows routinely from \cite[Corollary 21.23]{kechris95} that it can't be $K_\sigma$.)

\medskip

Finally, we mention the following open problem which goes back to Zhang (see e.g. \cite{kastermans09}). A permutation $\rho:\omega\to\omega$ is called \emph{cofinitary} if either $\rho=\id$ or $\{k\in\omega:\rho(k)=k\}$ is finite. A \emph{cofinitary group} is a subgroup of $S_\infty$ (= the group of \emph{all} permutations of $\omega$) which consists only of cofinitary permutations. $S_\infty$ is a Polish group in the relative topology it inherits from $\omega^\omega$, and so it makes sense to ask:

\begin{problem}
Can a maximal cofinitary group be analytic?
\end{problem}

Embarrassingly, we don't know if such a group can be \emph{closed}. Kastermanns \cite{kastermans09} has shown that such a group can't be $K_\sigma$. In a similar vein, we ask:

\begin{problem}
Are there maximal cofinitary groups in Solovay's model?
\end{problem}

\bibliographystyle{amsplain}
\bibliography{adfamilies}

\end{document}